\documentclass{article}

\usepackage{amsmath,amsthm,amssymb,amsfonts}

\usepackage{verbatim}
\usepackage{graphicx}

\usepackage{stmaryrd} 
\usepackage{hyperref}

\newtheorem{thm}{Theorem}

\newtheorem{prop}[thm]{Proposition}

\newtheorem{remarks}[thm]{Remark}

\newtheorem{definition}[thm]{Definition}

\newtheorem{exl}[thm]{Example}

\numberwithin{thm}{section}

\newcommand{\adj}{\leftrightarrow}

\newcommand{\adjeq}{\leftrightarroweq}

\def\Z{{\mathbb Z}}
\def\N{{\mathbb N}}

\def\R{{\mathbb R}}

\begin{document}

\title{Remarks on Fixed Point Assertions in Digital Topology, 5}

\author{Laurence Boxer
\thanks{Department of Computer and Information Sciences, Niagara University, NY 14109, USA;
and Department of Computer Science and Engineering, State University of New York at Buffalo \newline
email: boxer@niagara.edu \newline
Paper accepted for publication in {\em Applied General Topology}}
}

\date{ }
\maketitle

\begin{abstract}
As in~\cite{BxSt19,Bx19,Bx19-3,Bx20}, we 
discuss published assertions concerning fixed
points in ``digital metric spaces" - assertions
that are incorrect or incorrectly proven, 
or reduce to triviality. 


\end{abstract}

\section{Introduction}
As stated in~\cite{Bx19}:
\begin{quote}
The topic of fixed points in digital topology has drawn
much attention in recent papers. The quality of
discussion among these papers is uneven; 
while some assertions have been correct and interesting, others have been incorrect, incorrectly proven, or reducible to triviality.
\end{quote}
Paraphrasing~\cite{Bx19} slightly: 
in~\cite{BxSt19,Bx19,Bx19-3,Bx20}, we have discussed many shortcomings in earlier papers and have offered
corrections and improvements. We continue this work in the current paper.

Authors of many weak papers concerning 
fixed points in digital topology
seek to obtain results in a ``digital metric space" (see section~\ref{DigMetSp} for its definition).
This seems to be a bad idea. We quote~\cite{Bx20}:
\begin{quote}
\begin{itemize}
    \item Nearly all correct nontrivial 
    published assertions concerning digital
    metric spaces use either the adjacency of the digital image or the metric, but not both.
    \item If $X$ is finite (as in a ``real 
    world" digital image) or the metric $d$ 
          is a common metric such as any $\ell_p$ metric, then 
          $(X,d)$ is uniformly discrete as a topological space, 
          hence not very interesting.
    \item Many of the published assertions concerning
          digital metric spaces mimic analogues for subsets 
          of Euclidean~$\R^n$. Often, the authors neglect 
          important differences between the topological 
          space $\R^n$ and digital images, resulting in 
          assertions that are incorrect, trivial, or trivial
          when restricted to conditions that many regard as 
          essential. E.g., in many cases, functions that
          satisfy fixed point assertions must be constant or fail to be digitally continuous~\cite{BxSt19,Bx19,Bx19-3}.
\end{itemize}
\end{quote}

Since the publication of~\cite{Bx20}, additional
papers concerning fixed points in digital metric spaces
have come to our attention.
This paper continues the work 
of~\cite{BxSt19,Bx19,Bx19-3,Bx20} in discussing 
shortcomings of published assertions concerning fixed points in digital metric spaces.

Many of the definitions and assertions we
discuss were written with typographical
and grammatical errors,
and mathematical flaws. We have quoted 
these by using images of the originals
so that the reader can see these errors as they appear
in their sources (we have
removed or replaced with a different style
labels in equations and inequalities
in the images to remove confusion with
labels in our text).

\section{Preliminaries}
Much of the material in this section is quoted or
paraphrased from~\cite{Bx20}.

We use $\N$ to represent the natural numbers,
$\Z$ to represent the integers, and $\R$ to represent the reals.

A {\em digital image} is a pair $(X,\kappa)$, where $X \subset \Z^n$ 
for some positive integer $n$, and $\kappa$ is an adjacency relation on $X$. 
Thus, a digital image is a graph.
In order to model the ``real world," we usually take $X$ to be finite,
although there are several papers that consider
infinite digital images. The points of $X$ may be 
thought of as the ``black points" or foreground of a 
binary, monochrome ``digital picture," and the 
points of $\Z^n \setminus X$ as the ``white points"
or background of the digital picture.

For this paper, we need not specify
the details of adjacencies or of
digitally continuous functions.

A {\em fixed point} of a function $f: X \to X$ 
is a point $x \in X$ such that $f(x) = x$. 

\subsection{Digital metric spaces}
\label{DigMetSp}
A {\em digital metric space}~\cite{EgeKaraca15} is a triple
$(X,d,\kappa)$, where $(X,\kappa)$ is a digital image and $d$ is a metric on $X$. The
metric is usually taken to be the Euclidean
metric or some other $\ell_p$ metric.
We are not convinced that the
digital metric space is a notion 
worth developing.
Typically,
assertions in the literature do not make 
use of both $d$ and $\kappa$,
so that ``digital metric space" seems 
an artificial notion. E.g., for a
discrete topological space $X$, all 
functions $f: X \to X$ are continuous,
although on digital images, many 
functions $g: X \to X$ are not digitally continuous (digital continuity is defined
in~\cite{Bx99}, generalizing an earlier 
definition~\cite{Rosenfeld}).

\section{Assertions for contractions in~\cite{J-R18}}
The paper~\cite{J-R18} claims fixed point
theorems for several types of digital 
contraction functions. Serious errors
in the paper are discussed below.

\subsection{Fixed point for Kannan contraction}
Figure~\ref{fig:J-R-KannanDef} shows the definition
appearing in~\cite{J-R18} of a
{\em Kannan digital contraction}.

\begin{figure}
    \centering
    \includegraphics[height=1.5in]{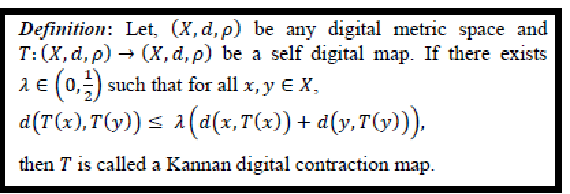}
    \caption{Definition of Kannan digital contraction in~\cite{J-R18}
          }
    \label{fig:J-R-KannanDef}
\end{figure}

\begin{figure}
    \centering
    \includegraphics[height=1in]{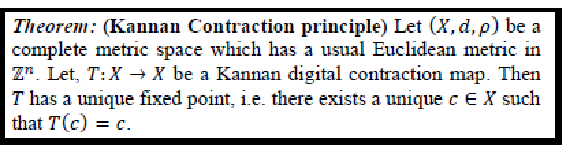}
    \caption{Fixed point assertion for Kannan digital contractions in~\cite{J-R18}
          }
    \label{fig:J-R-KannanAssert}
\end{figure}

Figure~\ref{fig:J-R-KannanAssert} shows a fixed point assertion
for Kannan digital contractions in~\cite{J-R18}. The ``proof"
of this assertion has errors discussed below.

\begin{itemize}
    \item In the fourth and fifth lines of the ``proof" is the claim that
\[ \lambda [d(x_n, x_{n-1}) + d(x_{n-1}, x_{n-2})] \le 
   2 \lambda d(x_n, x_{n-1}).
\]
Since $\lambda > 0$, this claim implies
\[ d(x_{n-1}, x_{n-2}) \le d(x_n, x_{n-1}),
\]
so if any $d(x_n, x_{n-1})$ is positive, the sequence
$\{x_n\}$ is not a Cauchy sequence, contrary to a claim that appears
later in the argument.
\item Towards the end of the existence argument, the authors claim
      that a Kannan digital contraction is
      digitally continuous. This assumption is contradicted
      by Example~4.1 of~\cite{BxSt19}.
\end{itemize}

In light of these errors, we must conclude that the assertion
of Figure~\ref{fig:J-R-KannanAssert} is unproven.

\subsection{Example of pp. 10769 - 10770}
This example is shown in
Figure~\ref{fig:J-R-exl-p10769}. One sees
easily the following.
\begin{itemize}
    \item $d(0,1)=d(0,2)=0$, so $d$, contrary
          to the claim, is not a metric.
    \item $T(1)=1/2 \not \in X$.
\end{itemize}

\begin{figure}
    \centering
    \includegraphics[height=1.5in]{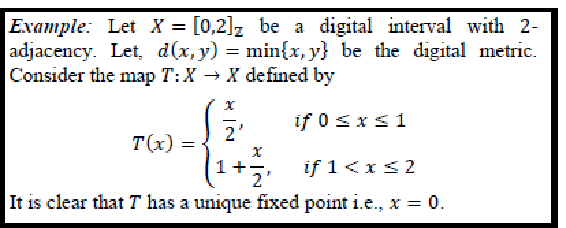}
    \caption{The example of~\cite{J-R18},
     pp. 10769 - 10770
          }
    \label{fig:J-R-exl-p10769}
\end{figure}

\subsection{Fixed point for generalization of
Kannan contraction}
Figure~\ref{fig:J-R-Kannan} shows an assertion
of a fixed point result on p. 10770
of~\cite{J-R18}.

\begin{figure}
    \centering
    \includegraphics[height=2.5in]{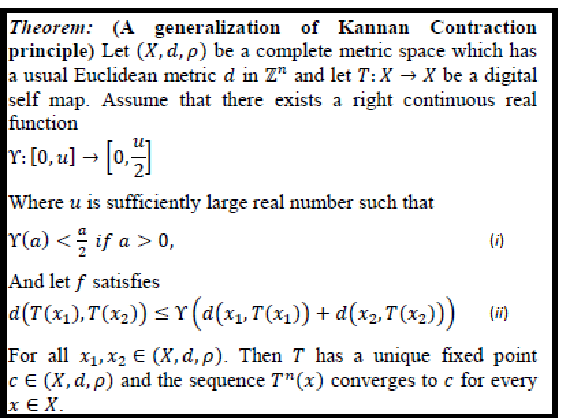}
    \caption{A ``theorem" of~\cite{J-R18},
     p. 10770
          }
    \label{fig:J-R-Kannan}
\end{figure}

Note ``And let $f$ satisfies" should be
``and let $T$ satisfy".

More importantly: 
In the argument offered as proof of this
assertion, the authors
let $x_0 \in X$, and, inductively,
\[ x_{n+1} = T(x_n),~~~~~a_{n+1} = d(x_n,x_{n+1}).
\]
They claim that by using the statements
marked (i) and (ii) in Figure~\ref{fig:J-R-Kannan}, it follows that
\[
 a_{n+1} = d(x_n,x_{n+1}) \le
   \Upsilon(d(x_n, x_{n-1}) + d(x_{n-1}, x_{n-2}))
\]
\begin{equation}
\label{incorrect}
   < 2d(x_n, x_{n-1}) = 2a_n,
\end{equation}
which, despite the authors' claim, does not
show that the sequence
$\{a_n\}$ is decreasing.
However, what correctly follows from the 
statements marked (i) and (ii) in Figure~\ref{fig:J-R-Kannan} is
\[
 a_{n+1} = d(x_n,x_{n+1}) =
 d(T(x_{n-1}),T(x_n)) \le
 \]
 \[
 \Upsilon(d(x_{n-1}, T(x_{n-1})) + d(x_n, T(x_n)))
= \Upsilon(d(x_{n-1}, x_n) + d(x_n, x_{n+1}))
\]
\begin{equation} 
\label{Y-ineq}
= \Upsilon(a_n + a_{n+1}) < \frac{a_n+a_{n+1}}{2}.
\end{equation}
From this we see that $a_{n+1} < a_n$, so
the sequence $\{a_n\}$ is decreasing and bounded below
by 0, hence tends to a limit $a \ge 0$.

The authors then claim that if $a > 0$ then
$a_{n+1} \le Y(a_n)$. However, what we showed
in~(\ref{Y-ineq}) does not support this
conclusion, which is not justified in any
obvious way. Since the authors wish to 
contradict the hypothesis that $a > 0$
in order to derive that the sequence
$\{x_n\}$ is a Cauchy sequence,
we must regard the assertion
shown in Figure~\ref{fig:J-R-Kannan}
as unproven.

\subsection{Fixed point for Zamfirescu
contraction}
\label{ZamfSec}
A {\em Zamfirescu digital contraction} is
defined in Figure~\ref{fig:MishraDef3.2}. This
notion is used in~\cite{J-R18,RanaGarg}, and will be
discussed in the current section and in
section~\ref{RanaGargSec}.

    \begin{figure}
    \centering
    \includegraphics[height=1in]{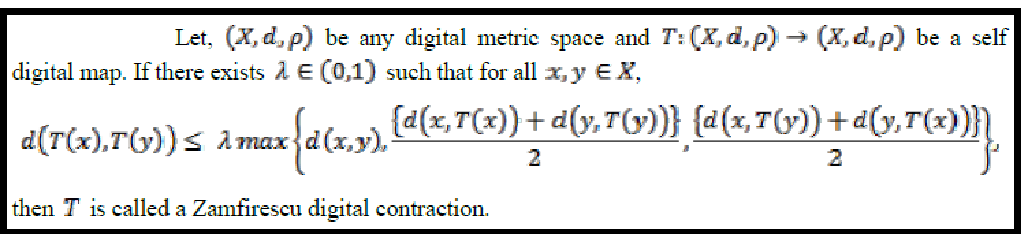}
    \caption{Definition~3.2 of~\cite{MishraEtAl},
    used in~\cite{J-R18,RanaGarg}.
          }
    \label{fig:MishraDef3.2}
\end{figure}

Figure~\ref{fig:J-R-Zamf} shows an
assertion found on p. 10770
of~\cite{J-R18}.

\begin{figure}
    \centering
    \includegraphics[height=1in]{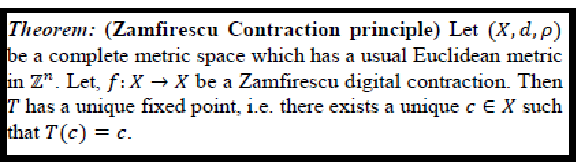}
    \caption{Another ``theorem" of~\cite{J-R18},
     p. 10770
          }
    \label{fig:J-R-Zamf}
\end{figure}

The argument offered as ``proof" of this 
assertion considers cases. For an arbitrary
$x_0 \in X$, a sequence is inductively
defined via $x_{n+1}=T(x_n)$. 
For convenience, let us define
\[ M(x,y) = 
\max \left \{ d(x,y), \frac{d(x,Tx) + d(y,Ty)}{2}, \frac{d(x,Ty) + d(y,Tx)}{2} \right \}.
\]
The argument considers several cases.
\begin{itemize}
    \item Case 1 says if
           $d(x_{n+1},x_n) = M(x_{n+1},x_n)$ then, by
           implied induction,
           \[ d(x_{n+1},x_n) \le
              \lambda^n d(x_1,x_0).
           \]
           But this argument is based
           on the unproven assumption
           that this is also the case
           for all indices $i < n$; i.e.,
           that $d(x_{i+1},x_i) = M(x_{i+1},x_i)$.
    \item Case 2 says if
\[ \frac{d(x_{n+1},Tx_{n+1}) + d(x_n,Tx_n)}{2} =
\]
\[ \max \left \{ \begin{array}{l}
d(x_{n+1},x_n), \frac{d(x_{n+1},Tx_{n+1}) + d(x_n,Tx_n)}{2}, \\ \\ \frac{d(x_{n+1},Tx_n) + d(x_n,Tx_{n_1})}{2}
 \end{array} \right \},
\]
then
\[ d(x_{n+1},x_n) = d(Tx_n, Tx_{n-1}) \le
   \lambda \frac{d(x_n,x_{n-1}) + d(x_{n-1},x_{n-2})}{2} \le
   \]
   \[
   \lambda d(x_n, x_{n-1}).
\]
But in order for the second inequality
in this chain to be true, we would need
$d(x_{n-1},x_{n-2}) \le d(x_n,x_{n-1})$,
and no reason is given to believe the
latter.
\item Case 3 says that if
       $\frac{d(x_{n+1},Tx_n) + d(x_n,Tx_{n+1})}{2} = M(x_{n+1}, x_n)$ then
       \[ d(x_{n+1},x_n) = d(Tx_n, Tx_{n-1}) \le
      \lambda \frac{d(x_n,x_{n-2}) + d(x_{n-1},x_{n-1})}{2}.
      \]
      The correct upper bound, according to the definition
      shown in Figure~\ref{fig:MishraDef3.2}, is
       \[ \lambda \frac{d(x_{n+1},Tx_n) + d(x_n,Tx_{n+1})}{2} = \lambda \frac{d(x_{n+1},x_{n+1}) + d(x_n,x_{n+2})}{2} =
       \]
       \[\lambda \frac{ d(x_n,x_{n+2})}{2}.
       \]
Further, the conclusion reached by the authors for
this case, that the distances $d(x_{n+1},x_n)$
are bounded above by an expression that tends
to 0 as $n \to \infty$, depends on the unproven
hypothesis that an analog of 
this case holds for all indices $i < n$.
\end{itemize}

Thus all three cases considered by the
authors are handled incorrectly. We
must conclude that the assertion of
Figure~\ref{fig:J-R-Zamf} is unproven.

\subsection{Fixed point for Rhoades contraction}
Figure~\ref{fig:J-R-RhoadesDef} shows the
definition appearing in~\cite{J-R18} of a
{\em Rhoades digital contraction}. The 
paper~\cite{J-R18} claims the fixed point result
shown in Figure~\ref{fig:J-R-RhoadesAssert}
for such functions. The argument offered in
``proof" of this assertion has errors that
are discussed below.

\begin{figure}
    \centering
    \includegraphics[height=1.25in]{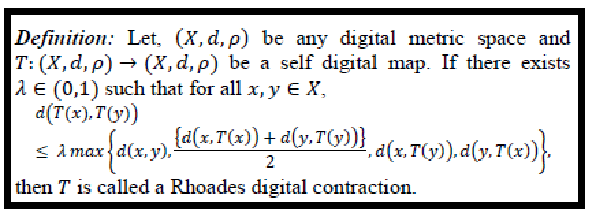}
    \caption{Definition of Rhoades digital contraction,~\cite{J-R18},
     p. 10769
          }
    \label{fig:J-R-RhoadesDef}
\end{figure}

\begin{figure}
    \centering
    \includegraphics[height=1in]{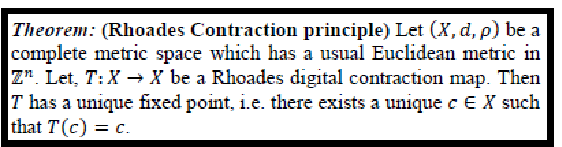}
    \caption{Fixed point assertion for Rhoades digital contraction,~\cite{J-R18},
     p. 10771
          }
    \label{fig:J-R-RhoadesAssert}
\end{figure}

For convenience, let
\[ M(x,y) = \max \left \{
    d(x,y),~\frac{d(x,Tx) + d(y,Ty)}{2},
       ~d(x,Ty),~d(y,Tx)
\right \}.
\]

The authors' argument considers cases corresponding
to which of the embraced expressions above gives
the value of $M(x_{n+1},x_n)$. In each case, the 
authors assume without proof that the same
case is valid for $M(x_{i+1},x_i)$, for all
indices $i < n$.

Additional errors:
\begin{itemize}
    \item In case 2, the inequality
          \[ d(Tx_n, Tx_{n-1}) \le 
             \lambda \frac{d(x_n,x_{n-1})+d(x_{n-1},x_{n-2})}{2}
          \]
          should be, according to Figure~\ref{fig:J-R-RhoadesDef},
          \[ d(x_n,x_{n+1}) = d(Tx_n, Tx_{n-1}) \le 
          \lambda \frac{d(x_n,Tx_n)+d(x_{n+1},Tx_{n-1})}{2}
          \]
          \[
             = \lambda \frac{d(x_n,x_{n+1})+d(x_{n+1},x_n)}{2}
             = \lambda d(x_n,x_{n+1}).
          \]
    Note this implies 
    \begin{equation}
    \label{eventualConst}
    x_n = x_{n+1},
    \end{equation}
    which would
    imply the existence of a fixed point.
    Also, the authors claim that
    \[ \lambda \frac{d(x_n,x_{n-1})+d(x_{n-1},x_{n-2})}{2} \le
     \lambda d(x_n,x_{n-1}),
    \]
    which is equivalent to
    \[ d(x_{n-1},x_{n-2}) \le d(x_n,x_{n-1}).
    \]
    No reason is given in support of the latter; further, it
    undermines the later claim that $\{x_n\}$ is a Cauchy
    sequence, since the authors did not 
    deduce~(\ref{eventualConst}).
    \item In case 3, it is claimed that
          \[ \lambda d(x_n,Tx_{n-1}) \le
             \lambda d(x_n,x_{n-1}).
          \]
          This should be corrected to
          \[ \lambda d(x_n,Tx_{n-1}) =
             \lambda d(x_n,x_n) = 0,
          \]
          which would guarantee a fixed point.
    \item In case 4, we see the claim
          \[ d(x_{n-1},Tx_n) = d(x_{n-1},x_{n-1}) = 0.
          \]
          This should be corrected to
          \[ d(x_{n-1},Tx_n) = d(x_{n-1},x_{n+1}).
          \]
\end{itemize}

In view of these errors, we must regard the assertion shown
in Figure~\ref{fig:J-R-RhoadesAssert} as unproven.

\section{Assertions for weakly compatible maps in~\cite{BarveEtAl}}
In this section, we show that
the assertions of~\cite{BarveEtAl} are
trivial or incorrect.

\subsection{Theorem~3.1 of~\cite{BarveEtAl}}
\begin{definition}
\label{WeaklyCompatDef}
{\rm \cite{Dalal17}}
Let $S,T: X \to X$. Then $S$ and $T$ are
{\em weakly compatible} or 
{\em coincidentally commuting} if, for every 
$x \in X$ such that $S(x)=T(x)$ we have 
$S(T(x)) = T(S(x))$.
\end{definition}

The assertion stated as Theorem~3.1
in~\cite{BarveEtAl} is shown in
Figure~\ref{fig:BarveTheorem3.1}. 
In this section, we show the assertion is 
false except in a trivial case.

Note if $d$ is any $\ell_p$ metric (including the usual 
Euclidean metric) then the requirements of
closed subsets of $X$ are automatically satisfied, since $(X,d)$ 
is a discrete space.

\begin{thm}
\label{noBarve}
If functions $G,H,P,Q$ satisfy the hypotheses of
Theorem 3.1 of~{\rm \cite{BarveEtAl}} then 
$G=H=P=Q$. 
Therefore, each of the pairs
$(P,G)$ and $(Q,H)$ has a unique
common point of coincidence 
if and only if $N$ consists of a single point.
\end{thm}

\begin{proof}
We observe the following.

\begin{figure}
    \centering
    \includegraphics[height=2.25in]{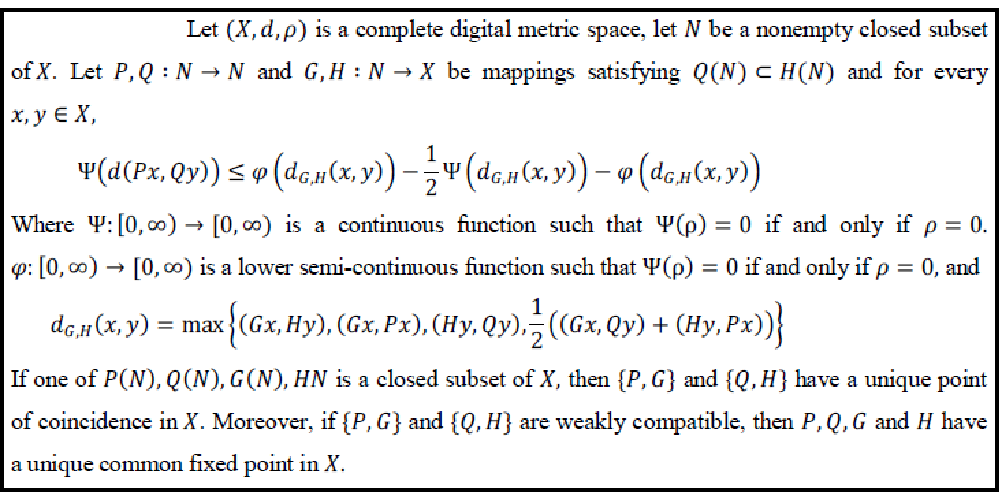}
    \caption{The assertion stated as
    Theorem 3.1 of~\cite{BarveEtAl}
          }
    \label{fig:BarveTheorem3.1}
\end{figure}

\begin{itemize}
    \item The inequality in the assertion simplifies as

 \[ \Psi(d(Px, Qy)) \le - \,\frac{1}{2}\Psi(d_{G,H}(x,y)).
    \]
    Since the function $\Psi$ is non-negative,
we have
\begin{equation}
\label{PsiZeroes}
    \Psi(d(Px, Qy)) = \Psi(d_{G,H}(x,y)) = 0.
\end{equation}
Therefore, we have
\begin{equation}
\label{PandQ}
     P(x) = Q(y) \mbox{ for all } x,y \in N.
\end{equation}
    \item In the equation for $d_{G,H}$ in
          Figure~\ref{fig:BarveTheorem3.1}, the pairs of
          points listed on the right side should be
          understood as having $d$ applied, i.e.,
          \begin{equation}
              \label{dGH}
           d_{G,H}(x,y) = \max \left \{
             \begin{array}{c}
             d(G(x),H(y)), d(G(x),P(x)), d(H(y),Q(y)), \\
             \frac{1}{2}[d(G(x),Q(y)) + d(H(y),P(x))]
             \end{array}
           \right \}.
          \end{equation}
Since $\Psi(x)=0$ if and only if $x=0$, we have
from~(\ref{PsiZeroes}) that
$d_{G,H}(x,y) = 0$, so (\ref{PandQ}) 
and (\ref{dGH}) imply
\[ G(x)=P(x)=Q(x)=H(x) \mbox{ for all } x \in N.
\]
\end{itemize}
We conclude that
$(P,G)$ and $(Q,H)$ are pairs of functions whose respective
common points of coincidence are unique 
if and only if $N$ consists of a single point.
\end{proof}

\subsection{Example~3.1 of~\cite{BarveEtAl}}
In Figure~\ref{fig:BarveExl3.1}, we see the assertion
presented as Example~3.1 of~\cite{BarveEtAl}.

\begin{figure}
    \centering
    \includegraphics[height=1.9in]{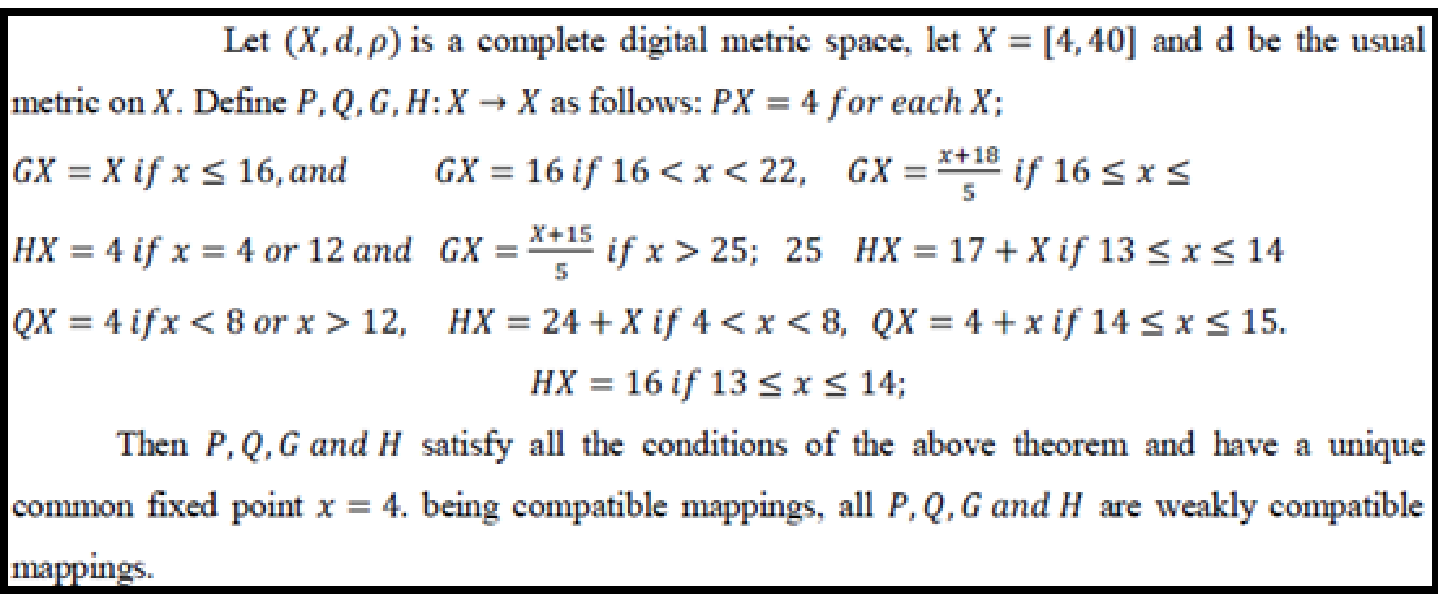}
    \caption{The assertion stated as
    Example~3.1 of~\cite{BarveEtAl}
          }
    \label{fig:BarveExl3.1}
\end{figure}

Note the following.
\begin{itemize}
    \item If ``$X=[4,40]$" is meant to mean the real interval
          from 4 to 40, then $X$ is not a digital image,
          as all coordinates of members of a digital image
          must be integers. Perhaps $X$ is meant to be the
          digital interval $[4,40]_{\Z} = [4,40] \cap \Z$.
    \item The function $G$ is not defined on all 
          of $[4,40]_{\Z}$, appears to be
          doubly defined for some values of $x$, 
          (notice the incomplete inequality at 
          the end of the 3rd line) and is not 
          restricted to integer values.
    \item The function $H$ is not defined on 
    all of $[4,40]_{\Z}$ and is doubly 
          defined for           $x \in \{13,14\}$.
    \item The function $Q$ is not defined for
          $x \in \{9,10,11,12\}$, and is doubly defined for $x \in \{14, 15\}$.
    \item The ``above theorem" referenced in the last sentence
          of Figure~\ref{fig:BarveExl3.1} is the assertion
          discredited by our Theorem~\ref{noBarve}, which shows
          that $P=Q=G=H$. Clearly, the assertion shown
          in Figure~\ref{fig:BarveExl3.1} fails to satisfy
          the latter.
\end{itemize}
Thus, Example~3.1 of~\cite{BarveEtAl} is not useful.

\subsection{Corollary 3.2 of~\cite{BarveEtAl}}
The assertion presented as Corollary~3.2 of~\cite{BarveEtAl}
is presented in Figure~\ref{fig:BarveCor3.2}. Notice that ``weakly mappings" is an 
undefined phrase. No proof is given
in~\cite{BarveEtAl} for this assertion, 
and it is not clear how the assertion 
might follow from previous assertions of 
the paper (which, as we have 
seen above, are also flawed).

Perhaps ``weakly mappings"
is intended to be ``weakly compatible 
mappings". At any rate, by labeling this 
assertion as a Corollary, the
authors suggest that it follows from the paper's
flawed ``Theorem"~3.1.

The assertion presented as ``Corollary"~3.2
of~\cite{BarveEtAl} must be regarded as undefined
and unproven.

    \begin{figure}
    \centering
    \includegraphics[height=0.75in]{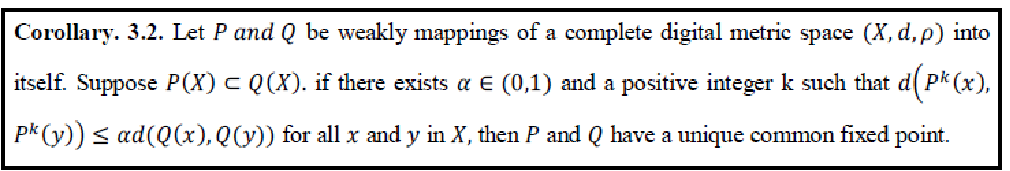}
    \caption{The assertion stated as
    Corollary~3.2 of~\cite{BarveEtAl}
          }
    \label{fig:BarveCor3.2}
\end{figure}

\section{Assertion for coincidence and fixed points
         in~\cite{MishraEtAl19}}
\begin{figure}
    \centering
    \includegraphics[width=4.5in]{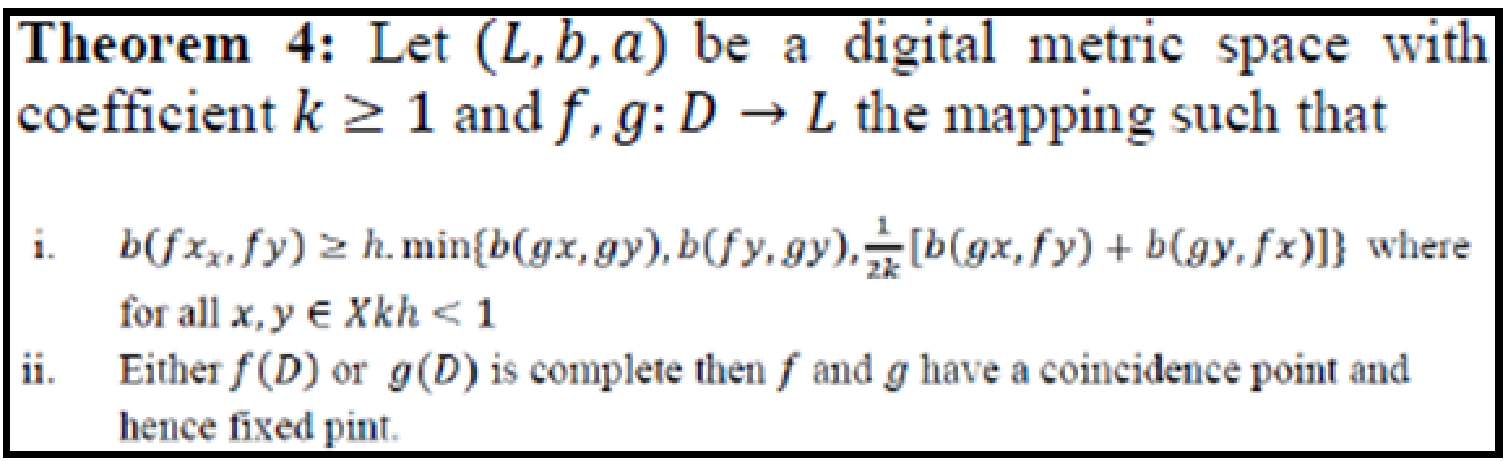}
    \caption{The assertion presented as Theorem 4
             of~\cite{MishraEtAl19}}
    \label{fig:MishraEtAl19}
\end{figure}
Figure~\ref{fig:MishraEtAl19} shows the assertion presented
as Theorem~4 of~\cite{MishraEtAl19}. The assertion as stated
is false. Flaws in this assertion include:
\begin{itemize}
    \item ``$D$" apparently should be ``$L$", and ``$x_x$"
          apparently should be ``$x$".
    \item No restriction is stated for the value of $h$. Therefore,
          we can take $h=0$, leaving the inequality in i) as
          $b(fx,fy) \ge 0$; since $b$ is a metric, this 
          inequality is not a restriction. Thus $f$ and $g$
          are arbitrary; they need not have a coincidence point
          or fixed points.
\end{itemize}

\section{Assertion for Zamfirescu contractions in~\cite{RanaGarg}}
\label{RanaGargSec}
Let $f: X \to X$, where
$(X,d,\kappa)$ is a digital metric space.
Recall that a {\em Zamfirescu digital
contraction}~\cite{MishraEtAl} is defined
in Figure~\ref{fig:MishraDef3.2}.

We show the assertion presented as Theorem~4.1
of~\cite{RanaGarg} in
Figure~\ref{fig:RanaThm4.1}.
    
        \begin{figure}
    \centering
    \includegraphics[height=1.6in]{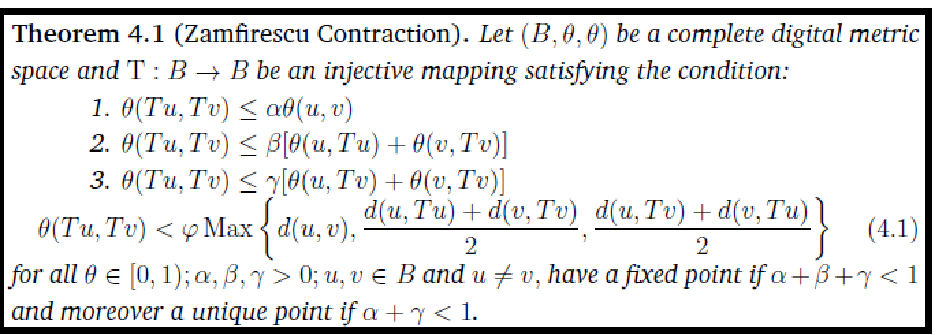}
    \caption{The assertion presented as Theorem~4.1 of~\cite{RanaGarg}
          }
    \label{fig:RanaThm4.1}
\end{figure}

Observe:
\begin{itemize}
    \item The symbol $\theta$ has distinct uses
          in this assertion. In the first line,
          $\theta$ is introduced as both the
          metric and the adjacency of $X$. Since
          our discussion below does not use
          an adjacency, we will assume $\theta$
          is the metric $d$.
    \item The symbol $\varphi$ seems intended to
          be a real number satisfying some
          restriction, but no restriction is
          stated. Alternately, it may be that
          $\varphi$ is intended to be a
          function to be applied to the
          $Max$ value in the statement, but
          no description of such a function
          appears.
\end{itemize}

Perhaps most important, we have the following.

\begin{thm}
If $Y$ has more than one point and
$d$ is any $\ell_p$ metric, then no
function $T$ satisfies the
hypotheses shown in Figure~\ref{fig:RanaThm4.1}.
\end{thm}

\begin{proof}
Suppose there is such a function $T$.
By choice of $d$, there exist $u_0,v_0 \in X$ such that
\[ d(u_0,v_0) = \min\{d(x,y)~|~x,y \in X,
      x \neq y\}.
\]
By the inequality stated in item~1
of Figure~\ref{fig:RanaThm4.1},
$d(Tu_0,Tv_0) = 0$. This contradicts the
hypothesis that $T$ is injective.
\end{proof}

\section{Assertion for expansive map in~\cite{ChauhanEtAl}}
The paper~\cite{ChauhanEtAl} claims to have
a fixed point theorem for digital metric
spaces. However, it is not
clear what the authors intend to assert,
as the paper has many undefined
and unreferenced terms and many obscuring
typographical errors.

The assertion stated as ``Preposition"~2.7
of~\cite{ChauhanEtAl} (and also as an
unlabeled Proposition on p. 10769 
of~\cite{J-R18}) was shown in Example~4.1
of~\cite{BxSt19} to be false.

\begin{figure}
    \centering
    \includegraphics[height=2in]{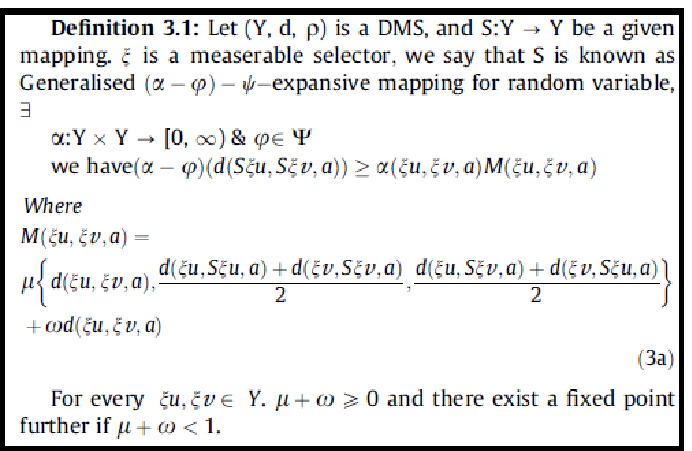}
    \caption{The statement presented as
    Definition~3.1 of~\cite{ChauhanEtAl}
          }
    \label{fig:ChauhanDef3.1}
\end{figure}

The definition of what this paper calls a
{\em Generalized} ($\alpha~-~\phi$)~-~$\psi$
{\em expansive mapping for random variable}
is shown in Figure~\ref{fig:ChauhanDef3.1}.
We observe the following.
\begin{itemize}
    \item This is not the same as a
          $\beta~-~\psi~-~\phi$ 
          {\em expansive mapping} defined
          in~\cite{Jyoti-Rani18}.
    \item Notice the $\rho$ that appears
          intended to be the adjacency of
          the digital metric space. This is
          significant in our discussion later.
    \item The set $\Psi$ is not defined anywhere
          in the paper. Perhaps it is meant
          to be the set $\Psi$ 
          of~\cite{Jyoti-Rani18}.
    \item The functions $d$, $\alpha$, and
          $M$ all have a third parameter $a$
          that appears to be extraneous, since
          each of these functions is used later
          with only two parameters.
    \item One supposes $\mu$ and $\omega$ must be
          non-negative, but this is not stated.
\end{itemize}

The assertion presented as Theorem~3.3
of~\cite{ChauhanEtAl} is shown in
Figure~\ref{fig:ChauhanThm3.3}.

\begin{figure}
    \centering
    \includegraphics[height=1.6in]{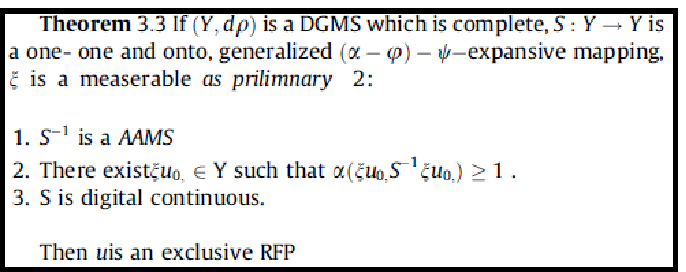}
    \caption{The assertion presented as
    Theorem~3.3 of~\cite{ChauhanEtAl}
          }
    \label{fig:ChauhanThm3.3}
\end{figure}

In the statement of this assertion:
\begin{itemize}
    \item ``DGMS" appears, although it is not
          defined anywhere in the paper. Perhaps
          it represents ``digital metric space".
    \item The term ``exclusive RFP" is not
          defined anywhere in the paper.
          One supposes the ``FP" is for
          ``fixed point".
\end{itemize}

In the argument offered as ``Verification"
of this assumption, we note the following.
\begin{itemize}
    \item The second line of the verification
          contains an undefined operator,
          ~\put(1,1){\shortstack{$+$ \\ $n$}}~
           \newline which perhaps is meant 
           to be $+$.
    \item The same line contains part of the
          phrase ``$u_n$ is a unique point 
          of $S$." What the authors intend by this
          is unclear.
    \item At the start of the long 
          statement~(3e), it is claimed that
          $M(\xi u_n, \xi u_{n+1})$ is the
          maximum of three expressions.
          The second term of the expression 
          for $M(\xi u_n, \xi u_{n+1})$ applies
          $\rho$ to a numeric expression. This makes no sense, 
          since $\rho$ is the adjacency of $Y$
          (see Figure~\ref{fig:ChauhanDef3.1}).
          Notice also that
          Figure~\ref{fig:ChauhanDef3.1}
          shows no such term in its expression
          for the function $M$. The use of 
          $\rho$, as a numeric value that has 
          neither been defined nor restricted
          to some range of values,
          propagates through both of
          the cases considered.
    \item In the expression for
          $M(\xi u_n, \xi u_{n+1})$,
          the third term, 
          $\omega(\xi u_n, \xi u_{n-1})$,
          should be
          $\omega d(\xi u_n, \xi u_{n-1})$ according
          to Figure~\ref{fig:ChauhanDef3.1}.
          This error repeats several times in
          statement (3e).
\end{itemize}

Other errors are present, but we have
established enough to conclude that whatever
the authors were trying to prove is unproven.

\section{Further remarks}
We have shown that nearly every assertion
introduced in the 
papers~\cite{J-R18,BarveEtAl,MishraEtAl19,RanaGarg,ChauhanEtAl}
is incorrect, unproven due to errors in the ``proofs,"
or trivial. These papers are
part of a larger body of highly flawed
publications devoted to fixed point assertions
in digital metric spaces, and emphasize our
contention that the digital metric space is
not a worthy subject of study.

\end{document}